\documentclass[12pt,oneside]{article}

\usepackage{amssymb}
\usepackage{amsmath}
\usepackage{tikz}
\usetikzlibrary{shapes,arrows,positioning}
\usepackage[letterpaper,left=1in,right=1in, bottom=1in, top=1in]{geometry}
\usepackage{titlesec}
\usepackage{setspace}
\usepackage{textcomp}
\usepackage{amsthm}
\usepackage{eso-pic}
\usepackage{graphicx}
\usepackage{color}
\usepackage{longtable}
\usepackage{transparent}
\usepackage{caption}
\usepackage{hyperref}
\usepackage{algorithm}
\usepackage{algorithmic}
\usepackage{placeins}

\renewcommand{\leq}{\leqslant}

\renewcommand{\cdots}{\ldots}

\newcommand{\tab}{\hspace{1cm}}

\newcommand{\su}{\subseteq}

\newcommand{\bc}{\bigcup}
\newcommand{\cd}{\cdots}

\titleformat{\section}{\bfseries\centering}{\thesection \hspace{0.5cm}}{12pt}{}

\titleformat{\subsection}{\bfseries}{\thesubsection \hspace{0.5cm}}{12pt}{}

\newtheorem{den}{Definition}[section]

\newtheorem{thm}[den]{Theorem}
\newtheorem{lem}[den]{Lemma}
\newtheorem{cor}[den]{Corollary}
\newtheorem{pro}[den]{Proposition}
\newtheorem{con}[den]{Conjecture}

\begin{document}\setlength{\parindent}{0cm}\thispagestyle{empty}
	\begin{center}
		\Large \textbf{On Legendre Cordial Labeling of Complete Graphs}
	\end{center}
	\vspace{2mm}
	
	\textbf{Jason D. Andoyo}
	
	University of Southeastern Philippines, Davao City, Philippines
	
	\section*{Abstract} \small
	\tab Let $p$ be an odd prime. For a simple connected graph $G$ of order $n$, a bijective function $f:V(G)\to\{1,2,\cd,n\}$ is said to be a Legendre cordial labeling modulo $p$ if the induced function $f_p^*:E(G)\to \{0,1\}$, defined by $f_p^* (uv)=0$ whenever $([f(u)+f(v)]/p)=-1$ or $f(u)+f(v)\equiv 0(\text{mod } p)$ and $f_p^* (uv)=1$ whenever $([f(u)+f(v)]/p)=1$, satisfies the condition $|e_{f_p^*}(0)-e_{f_p^*}(1)|\leq 1$ where $e_{f_p^*}(i)$ is the number of edges with label $i$ ($i=0,1$). This paper explores the characterization of the Legendre cordial labeling modulo $p$ of the complete graph $K_n$ using the concept of Legendre graph. 
	
	\vspace{2mm}
	
	\textbf{Keywords:} odd prime, Legendre cordial labeling, complete graph, Legendre graph
	
	\vspace{2mm}
	
	\textbf{2020 Mathematics Subject Classification:} 05C78, 11A07, 11A15
	
	\normalsize

	\section{Introduction}
	\tab A graph $G=(V,E)$ consists of a \textit{vertex set} $V=V(G)$ and an \textit{edge set} $E=E(G)$. The elements of $V$ and $E$ are called \textit{vertices} and \textit{edges}, respectively. If $|V|=n$ and $|E|=m$, then $G$ has \textit{order} $n$ and \textit{size} $m$. For $v\in V(G)$, the \textit{open neighborhood} of $v$ is the set $N(v)=\{u:uv\in E(G)\}$ and the degree of $v$ is $\deg(v)=|N(v)|$. The \textit{minimum degree} of $G$ is $\delta(G)=\min\{\deg(v):v\in V(G)\}$. Similarly, the \textit{maximum degree} of $G$ is $\Delta(G)=\max\{\deg(v):v\in V(G)\}$. For further graph-theoretic notions, refer to \cite{Chartrand}.
	 
	\tab Graph theory, an important area of discrete mathematics, studies the relationships between vertices and edges in modeling complex relational structures. A cornerstone of this field is the fundamental theorem of graph theory, which states that the sum of the degrees of all vertices in a graph equals twice its size. This basic result shows the natural connection between local vertex properties and the overall structure of a graph, serving as a starting point for many further studies and applications. Notable applications of graph theory include its use in network design, social interaction modeling, and transport optimization.
	
	\tab In a similar way, number theory, a key area of pure mathematics, focuses on the properties of integers and their arithmetic relationships. It has well-known applications in areas such as cryptography, coding theory, and primality testing. Among its central tools is the Legendre symbol, first defined by Legendre in his 1798 work \textit{Essai sur la théorie des nombres}. An integer $a$, relatively prime to an odd prime $p$, is called a \textit{quadratic residue of $p$} if the quadratic congruence $x^2\equiv a(\text{mod } p)$ has a solution; otherwise $a$ is a \textit{quadratic nonresidue of $p$}. The \textit{Legendre symbol} $(a/p)$ is defined as follows:
	\begin{equation*}
		(a/p)=\begin{cases}
			1 &\text{ if $a$ is a quadratic residue of $p$}\\
			-1 &\text{ if $a$ is a quadratic nonresidue of $p$}.
		\end{cases}
	\end{equation*}
	One important property of the Legendre symbol is $(a/p)=(b/p)$ if $a$ and $b$ are integers relatively prime to odd prime $p$ with $a\equiv b(\text{mod } p)$. \cite{Burton}
	
	\tab One of the fundamental concepts in graph theory is graph labeling, which serves as a bridge between graph theory and number theory. Graph labeling refers to the assignment of numbers (or, in some cases, sets) to the vertices and/or edges of a graph $G$ under specified conditions \cite{Gallian}. Among its many applications, three of the most prominent are in coding theory, communication networks, and circuit design.
	
	\tab Cordial labeling, introduced by Cahit \cite{Cahit} in 1987, was inspired by Rosa’s \cite{Rosa} graceful labeling and represents a broader form of vertex labeling. In this scheme, each vertex of a graph G is assigned a label from {0,1}, and each edge is then labeled according to the sum of its endpoint labels taken modulo 2, that is, an edge receives label 0 if its endpoints have the same label, and label 1 if its endpoints have different labels. The labeling is called cordial if the numbers of vertices labeled 0 and 1 differ by at most one, and likewise the numbers of edges labeled 0 and 1 differ by at most one. Recently introduced variants of cordial labeling include topological cordial labeling by Selestin Lina et al. \cite{Selestin} and Möbius cordial labeling by Asha Rani et al. \cite{Asha}.
	
	\tab Another recent variation of cordial labeling is Legendre cordial labeling, introduced in \cite{Andoyo1}. Several classes of graphs, including path graphs, cycle graphs, and star graphs, have been studied. However, no research has been conducted on complete graphs. Therefore, this paper investigates the Legendre cordial labeling of complete graphs.
	
	\tab Let $p$ be an odd prime. For a simple connected graph $G$ of order $n$ define $f:V(G)\to \{1,2,\cd,n\}$ as a bijective function. Then $f$ is called \textit{Legendre cordial labeling modulo $p$} if the induced function $f_p^*:E(G)\to \{0,1\}$, defined by
	\begin{equation*}
		f_p^*(uv)=\begin{cases}
			0 &\text{ if $([f(u)+f(v)]/p)=-1$ or $f(u)+f(v)\equiv 0(\text{mod }p)$}\\
			1 &\text{ if $([f(u)+f(v)]/p)=1$,}
		\end{cases}
	\end{equation*}
	satisfies the condition $|e_{f_p^*}(0)-e_{f_p^*}(1)|\leq 1$ where $e_{f_p^*}(i)$ is the number of edges with label $i$ ($i=0,1$). A graph that admits a Legendre cordial labeling modulo $p$ is called \textit{Legendre cordial graph modulo $p$}.
	
	\tab The Möbius graph, introduced by Vasumathi and Vangipuram \cite{Vasumathi}, is defined on the vertex set $V=\{1,2,\cd,n\}$ with edge set $E=\{uv:\mu(uv)=0\}$ where $\mu$ denotes the Möbius function. This construction illustrates how number-theoretic functions can determine adjacency in graphs. Inspired by this idea, we define the Legendre graph, where adjacency is instead governed by the Legendre symbol. The Legendre graph thus extends the Möbius framework by incorporating quadratic residue conditions into the graph structure.
	
	\tab Let $p$ be an odd prime and $k\in \{-1,1\}$. A \textit{Legendre graph} $L_n^k(f,p)$ of order $n$ is a graph with edge set
	\begin{equation*}
		E(L_n^k(f,p))=\{ab:([f(a)+f(b)])=k\}
	\end{equation*}
	where $f:V(L_n^k(f,p))\to \{1,2,\cd,n\}$ is a bijective function.
	
	\section{Basic Concepts}
	\begin{den}
		A complete graph $K_n$ of order $n$ is a graph where two distinct vertices are adjacent. The size of $K_n$ is $\frac{n(n-1)}{2}$.
	\end{den}
	
	\begin{thm}[\cite{Burton}]\label{2.3}
		If $p$ is an odd prime, then $p$ has $\frac{p-1}{2}$ quadratic residues and $\frac{p-1}{2}$ quadratic nonresidues of $p$ in $\{1,2,...,p-1\}$.
	\end{thm}
	
	\section{Main Results}
	
	\tab Throughout the discussion, $p$ is denoted as an odd prime.
	\begin{lem}\label{3.2}
		Let $L_n^k(f,p)$ be a Legendre graph of order $n$ and let $q=\left\lfloor\frac{n}{p}\right\rfloor$. Suppose that $v\in V(L_n^k(f,p))$ and consider the following:\\
		for $1\leq f(v)\leq qp$, define
		\begin{equation*}
			\omega^k(v)=\bc_{c\in F}\{\xi:f(v)+c\equiv \xi(\text{\normalfont mod }p), \text{ }0\leq \xi<p,\text{ }(\xi/p)=k\},
		\end{equation*}
		for $qp+1\leq f(v)\leq n$, define
		\begin{equation*}
			\pi^k(v)=\bc_{c\in F-{\varepsilon(v)}}\{\xi:f(v)+c\equiv \xi(\text{\normalfont mod }p), \text{ }0\leq \xi<p,\text{ }(\xi/p)=k\},
		\end{equation*}
		where $f(v)\equiv \varepsilon(v)(\text{\normalfont{mod} }p)$, $1\leq \varepsilon(v)\leq n-qp$, and $F=\{1,2,\cd,n-qp\}$. Then\\
		for $1\leq f(v)\leq qp$,
		\begin{equation}\label{eq1}
			\deg(v)=\begin{cases}
				q\left(\frac{p-1}{2}\right)-1+|\omega^k(v)|&\text{ if }(2f(v)/p)=k\\
				q\left(\frac{p-1}{2}\right)+|\omega^k(v)|&\text{ if }(2f(v)/p)=-k\text{ or }f(v)\equiv 0(\text{\normalfont{mod }}p),
			\end{cases}
		\end{equation}
		for $qp+1\leq f(v)\leq n$,
		\begin{equation}\label{eq2}
			\deg(v)=q\left(\frac{p-1}{2}\right)+|\pi^k(v)|.
		\end{equation}
		Consequently, the size of $L_n^k(f,p)$ is
		\begin{equation*}
			\frac{1}{4}\left[n^2-n-2nq+pq^2+q-\psi+k(\mathcal{S}_1+\mathcal{S}_2)\right]
		\end{equation*}
		where $\psi=\max\{0,2(n-qp)-p+1\}$, and\\
		\begin{equation*}
			\mathcal{S}_1=\sum_{\substack{s=2\\s\neq p}}^{n-qp+1}(s-1-\delta_s)(s/p)\text{ and }\mathcal{S}_2=\sum_{\substack{s=n-qp+2\\s\neq p}}^{2(n-qp)}[2(n-qp)-s+1-\delta_s](s/p)
		\end{equation*}	
		given that
		\begin{equation*}
			\delta_s=\begin{cases}
				1&\text{ if $s$ is even}\\
				0&\text{ otherwise}.
			\end{cases}
		\end{equation*}
	\end{lem}
	
	\begin{proof}
		Suppose that $k=1$. Observe that the range of $f$, $R_f=\{1,2,\cd,n\}$, can be decomposed into
		\begin{equation*}
			R_f=\left[\bc_{r=0}^{q-1}\lambda_r\right]\cup\Lambda
		\end{equation*}
		where $\lambda_r=\{rp+1,rp+2,\cd,rp+p\}$ and $\Lambda=\{qp+1,qp+2,\cd,n\}$. In addition, define
		\begin{equation*}
			D(\Theta)=\{a\in V(L_n^1(f,p)):f(a)\in \Theta\}
		\end{equation*}
		for $\Theta=\lambda_0,\lambda_1,\cd,\lambda_{q-1},\Lambda$. Let $v\in V(L_n^1(f,p))$. It should be noted that $uv\in E(V(L_n^1(f,p)))$ if and only if $([f(u)+f(v)]/p)=1$.
		
		\textbf{Case 1.} Let $\Gamma$ be a set such that $v\in \Gamma$, $\Gamma\in \{\lambda_0,\lambda_1,\cd,\lambda_{q-1}\}\cup\{\Lambda\}$.
		
		\textbf{Claim 1.} $v$ is adjacent to $\frac{p-1}{2}$ vertices of $D(\lambda_s)$, $\lambda_s\neq \Gamma$
		
		Consider the set $\lambda_s$, $0\leq s\leq q-1$, so that $\lambda_s\neq \Gamma$. Let $u\in D(\lambda_s)$. Define the set
		\begin{equation*}
			\alpha_v=\bc_{u\in D(\lambda_s)}\{\xi:f(u)+f(v)\equiv \xi(\text{mod }p),\text{ }0\leq \xi<p\}.
		\end{equation*}
		If $f(v)\equiv \theta(\text{mod }p)$, $0\leq \theta<p$, then
		\begin{equation*}
			\alpha_v=\{1+\theta, 2+\theta,\cd,p-1,0,1,\cd,\theta\}=\{0\}\cup\{1,2,\cd,p-1\}.
		\end{equation*}
		It is evident that $p$ has $\frac{p-1}{2}$ quadratic residues in $\alpha_v$ according to Theorem~\ref{2.3}. This proves Claim 1.
		
		\textbf{Case 2.} Let $1\leq f(v)\leq qp$ and suppose that $v\in D(\lambda_r)$, $0\leq r\leq q-1$.
		
		\textbf{Claim 2.} $v$ is adjacent to $\frac{p-1}{2}$ vertices of $D(\lambda_r)$ whenever $(2f(v)/p)=-1$ or $f(v)\equiv 0(\text{mod }p)$
		
		\textbf{Claim 3.} $v$ is adjacent to $\frac{p-3}{2}$ vertices of $D(\lambda_r)$ whenever $(2f(v)/p)=1$
		
		Let $D(\lambda_r)$, $u\neq v$. Assume that
		\begin{equation*}
			\beta_v=\bc_{\substack{u\in D(\lambda_r)\\u\neq v}}\{\xi:f(u)+f(v)\equiv\xi(\text{mod }p),\text{ }0\leq \xi<p\}.
		\end{equation*}
		If $f(v)\equiv \theta(\text{mod }p)$ and $2f(v)\equiv \rho(\text{mod }p)$, $0\leq \theta,\rho<p$, then
		\begin{equation*}
			\beta_v=\{1+\theta,2+\theta,\cd,p-1,0,1,\cd,\theta\}-\{\rho\}=\{0\}\cup \{1,2,\cd,p-1\}-\{\rho\}.
		\end{equation*}
		Note that $p$ has $\frac{p-1}{2}$ quadratic residues in $\{1,2,\cd,p-1\}$ according to Theorem~\ref{2.3}. Also, note that $(2f(v)/p)=(\rho/p)$. So, if $(2f(v)/p)=-1$ or $f(v)\equiv 0(\text{mod }p)$, then $p$ has $\frac{p-1}{2}$ quadratic residues in $\beta_v$. This proves Claim 2. Furthermore, if $(2f(v)/p)=1$, then $p$ has $\frac{p-3}{2}$ quadratic residues in $\beta_v$. This proves Claim 3.
		
		\textbf{Claim 4.} $v$ is adjacent to $|\omega^1(v)|$ vertices of $D(\Lambda)$
		
		Let $u\in D(\Lambda)$ and consider the set
		\begin{equation*}
			\gamma_v=\bc_{u\in D(\Lambda)}\{\xi:f(v)+f(u)\equiv \xi(\text{mod }p),\text{ }0\leq \xi<p\}.
		\end{equation*}
		It is evident that $F=\bc_{u\in D(\Lambda)}\{\varpi:f(u)\equiv\varpi(\text{mod }p),\text{ }0\leq \varpi<p\}$. Hence,
		\begin{equation*}
			\gamma_v=\bc_{c\in F}\{\xi:f(v)+c\equiv \xi(\text{mod }p),\text{ }0\leq \xi<p\}.
		\end{equation*}
		Clearly, $\omega^1(v)=\{\varpi:\varpi\in \gamma_v \text{ and }(\varpi/p)=1\}$. This proves Claim 4.
		
		\textbf{Case 3.} Let $qp+1\leq f(v)\leq n$ and suppose that $v\in D(\Lambda)$.
		
		\textbf{Claim 5.} $v$ is adjacent to $|\pi^1(v)|$ vertices of $D(\Lambda)$
		
		Let $u\in D(\Lambda)$, $u\neq v$, and consider the set
		\begin{equation*}
			\eta_v=\bc_{\substack{u\in D(\Lambda)\\u\neq v}}\{\xi:f(v)+f(u)\equiv \xi(\text{mod }p),\text{ }0\leq \xi< p\}.
		\end{equation*}
		Because $F=\bc_{u\in D(\Lambda)}\{\varpi:f(u)\equiv\varpi(\text{mod }p),\text{ }0\leq \varpi<p\}$ as noted in Case 4, we have
		\begin{equation*}
			\eta_v=\bc_{c\in F-\{\varepsilon(v)\}}\{\xi:f(v)+c\equiv \xi(\text{mod }p),\text{ }0\leq \xi< p\}.
		\end{equation*}
		Clearly, $\pi^1(v)=\{\varpi:\varpi\in \eta_v\text{ and }(\varpi/p)=1\}$. This proves Claim 5.
		
		Combining Claims 1 to 5, we obtain (\ref{eq1}) and (\ref{eq2}) for $k=1$. 
		
		For the size of $L_n^1(f,p)$, let
		\begin{equation*}
			A=\sum_{\substack{1\leq f(v)\leq qp\\(2f(v)/p)=1}}\deg(v),
		\end{equation*}
		\begin{equation*}
			B=\sum_{\substack{1\leq f(v)\leq qp\\(2f(v)/p)=-1}}\deg(v),
		\end{equation*}	
		\begin{equation*}
			C=\sum_{qp+1\leq f(v)\leq n}\deg(v),
		\end{equation*}
		and
		\begin{equation*}
			D=\sum_{f(v)\equiv 0(\text{mod }p)}\deg(v).
		\end{equation*}
		So, $\sum_{v\in V(L_n^1(f,p))}\deg(v)=A+B+C+D$.
		
		For each $t=0,1,\cd,q-1$, consider the set
		\begin{equation*}
			\vartheta_t=\bc_{tp+1\leq f(v)\leq tp+p-1}\{\sigma:2f(v)\equiv \sigma(\text{mod }p),\text{ }0\leq \sigma<p\}
		\end{equation*}
		and it follows that
		\begin{equation*}
			\vartheta_t=\{2,4,\cd,p-1,1,3,\cd,p-2\}=\{1,2,\cd,p-1\}.
		\end{equation*} 
		So, $p$ has $\frac{p-1}{2}$ quadratic residues and $\frac{p-1}{2}$ quadratic nonresidues in $\vartheta_t$ in accordance with Theorem~\ref{2.3}, for $t=0,1,\cd,q-1$. Consequently, $p$ has $q\left(\frac{p-1}{2}\right)$ quadratic residues and $q\left(\frac{p-1}{2}\right)$ quadratic nonresidues in the set $\{1,2,\cd,qp\}$. Therefore,
		\begin{equation*}
			\sum_{\substack{1\leq f(v)\leq qp\\(2f(v)/p)=1}}\left[q\left(\frac{p-1}{2}\right)-1\right]=q^2\left(\frac{p-1}{2}\right)^2-q\left(\frac{p-1}{2}\right)
		\end{equation*}
		and
		\begin{equation*}
			\sum_{\substack{1\leq f(v)\leq qp\\(2f(v)/p)=-1}}q\left(\frac{p-1}{2}\right)=q^2\left(\frac{p-1}{2}\right)^2.
		\end{equation*}
		For $A$,
		\begin{equation*}
			A=\sum_{\substack{1\leq f(v)\leq qp\\(2f(v)/p)=1}}\left[q\left(\frac{p-1}{2}\right)-1+|\omega^1(v)|\right]=q^2\left(\frac{p-1}{2}\right)^2-q\left(\frac{p-1}{2}\right)+\sum_{\substack{1\leq f(v)\leq qp\\(2f(v)/p)=1}}|\omega^1(v)|.
		\end{equation*}
		For $B$,
		\begin{equation*}
			B=\sum_{\substack{1\leq f(v)\leq qp\\(2f(v)/p)=-1}}\left[q\left(\frac{p-1}{2}\right)+|\omega^1(v)|\right]=q^2\left(\frac{p-1}{2}\right)^2+\sum_{\substack{1\leq f(v)\leq qp\\(2f(v)/p)=-1}}|\omega^1(v)|.
		\end{equation*}
		This means that
		\begin{equation*}
			A+B=q^2\left(\frac{p-1}{2}\right)^2-q\left(\frac{p-1}{2}\right)+\sum_{\substack{1\leq f(v)\leq qp\\f(v)\not\equiv0(\text{mod }p)}}|\omega^1(v)|.
		\end{equation*}
		For $C$,
		\begin{equation*}
			C=\sum_{qp+1\leq f(v)\leq n}\left[q\left(\frac{p-1}{2}\right)+|\pi^1(v)|\right]=(n-qp)\left[q\left(\frac{p-1}{2}\right)\right]+\sum_{qp+1\leq f(v)\leq n}|\pi^1(v)|.
		\end{equation*}
		For $D$,
		\begin{equation*}
			D=\sum_{f(v)\equiv 0(\text{mod }p)}\left[q\left(\frac{p-1}{2}\right)+|\omega^1(v)|\right]=q\left[q\left(\frac{p-1}{2}\right)\right]+\sum_{f(v)\equiv 0(\text{mod }p)}|\omega^1(v)|.
		\end{equation*}
		
		Consequently,
		\begin{align*}
			\sum_{v\in V(L_n^1(f,p))}\deg(v)&=q\left(\frac{p-1}{2}\right)\left[2q\left(\frac{p-1}{2}\right)-1+n-qp+q\right]+\sum_{1\leq f(v)\leq qp}|\omega^1(v)|\\
			&\text{\hspace{5mm}}+\sum_{qp+1\leq f(v)\leq n}|\pi^1(v)|.
		\end{align*}
		\textbf{Claim 6.} $$\sum_{1\leq f(v)\leq qp}|\omega^1(v)|=q(n-qp)\left(\frac{p-1}{2}\right)$$
		
		Let $v\in V(L_n^1(f,p))$ with $tp+1\leq f(v)\leq tp+p$ for $t=0,1,\cd,q-1$. Define the set
		\begin{equation*}
			Z_v=\bc_{c\in F}\{\varrho:f(v)+c\equiv \varrho(\text{mod }p),\text{ }0\leq \varrho<p\}.
		\end{equation*}
		Observe that $\omega^1(v)\su Z_v$. Let $f(v)\equiv \xi(\text{mod }p)$ and $f(v)+c\equiv \varrho(\text{mod }p)$, $0\leq \xi,\varrho<p$. Consider Table~\ref{rho}.
		\begin{table}[!hbt]
			\centering
			\begin{tabular}{|c|c|c|c|c|c|c|c|c|}
				\hline
				\color{blue}$\xi$  \color{black}$|$  
				\color{red}$c$&\color{red}1&\color{red}2&\color{red}3&\color{red}$\cd$&\color{red}$n-qp$\\
				\hline
				\color{blue}1&2&3&4&$\cd$&$n-qp+1$\\
				\hline
				\color{blue}2&3&4&5&$\cd$&$n-qp+2$\\
				\hline
				\color{blue}3&4&5&6&$\cd$&$n-qp+3$\\
				\hline
				\color{blue}$\vdots$&$\vdots$&$\vdots$&$\vdots$&$\vdots$&$\vdots$\\
				\hline
				\color{blue}$(q+1)p-n-1$&$(q+1)p-n$&$(q+1)p-n+1$&$(q+1)p-n+2$&$\cd$&$p-1$\\
				\hline
				\color{blue}$(q+1)p-n$&$(q+1)p-n+1$&$(q+1)p-n+2$&$(q+1)p-n+3$& $\cd$&$0$\\
				\hline
				\color{blue}$(q+1)p-n+1$&$(q+1)p-n+2$&$(q+1)p-n+3$&$(q+1)p-n+4$& $\cd$&$1$\\
				\hline
				\color{blue}$\vdots$&$\vdots$&$\vdots$&$\vdots$&$\vdots$&$\vdots$\\
				\hline
				\color{blue}$p-2$&$p-1$&0&1&$\cd$&$n-qp-2$\\
				\hline
				\color{blue}$p-1$&0&1&2&$\cd$&$n-qp-1$\\
				\hline
				\color{blue}$p$&1&2&3&$\cd$&$n-qp$\\
				\hline
			\end{tabular}
			\caption{Values of $\varrho$}\label{rho}
		\end{table}
		
		The rows of Table~\ref{rho} represent the set $Z_v$. Also, notice that the columns produce the set $\{0\}\cup\{1,2,\cd,p-1\}$. Because $p$ has $\frac{p-1}{2}$ quadratic residues in $\{1,2,\cd,p-1\}$ in accordance with Theorem~\ref{2.3}, we have
		\begin{equation*}
			\sum_{tp+1\leq f(v)\leq tp+p}\omega^1(v)=(n-qp)\left(\frac{p-1}{2}\right)
		\end{equation*}
		for $t=0,1,\cd,q-1$. This proves Claim 6. 
		
		Therefore, 
		\begin{align*}
			\sum_{v\in V(L_n^1(f,p))}\deg(v)&=q\left(\frac{p-1}{2}\right)\left[2q\left(\frac{p-1}{2}\right)-1+n-qp+q+n-qp\right]+\sum_{qp+1\leq f(v)\leq n}|\pi^1(v)|\\
			&=q\left(\frac{p-1}{2}\right)\left(qp-q+2n-2qp+q-1\right)+\sum_{qp+1\leq f(v)\leq n}|\pi^1(v)|\\
			&=\frac{q(p-1)(2n-qp-1)}{2}+\sum_{qp+1\leq f(v)\leq n}|\pi^1(v)|.
		\end{align*}
		Finally, by fundamental theorem of graph theory, we have
		\begin{equation*}
			|E(L_n^1(f,p))|=\frac{q(p-1)(2n-qp-1)}{4}+\frac{1}{2}\sum_{qp+1\leq f(v)\leq n}|\pi^1(v)|.
		\end{equation*}
		For $k=-1$, the approach is analogous.
		
		\textbf{Claim 7.} 
		$$\sum_{qp+1\leq f(v)\leq n}|\pi^k(v)|=\frac{1}{2}[(n-qp)(n-qp-1)-\psi+k(\mathcal{S}_1+\mathcal{S}_2)]$$
		where $\psi=\max\{0,2(n-qp)-p+1\}$, and
		\begin{equation*}
		\mathcal{S}_1=\sum_{\substack{s=2\\s\neq p}}^{n-qp+1}(s-1-\delta_s)(s/p)\text{ and }\mathcal{S}_2=\sum_{\substack{s=n-qp+2\\s\neq p}}^{2(n-qp)}[2(n-qp)-s+1-\delta_s](s/p)
		\end{equation*}	
		given that
		\begin{equation*}
			\delta_s=\begin{cases}
				1&\text{ if $s$ is even}\\
				0&\text{ otherwise}
			\end{cases}
		\end{equation*}
		Let $T=\sum_{qp+1\leq f(v)\leq n}|\pi^k(v)|$. Then
		\begin{equation*}
			T=\sum_{qp+1\leq f(v)\leq n}\sum_{\substack{c\in F\\c\neq \varepsilon(v)\\f(v)+c\not\equiv0(\text{mod }p)}}\frac{1+k([f(v)+c]/p)}{2}.
		\end{equation*}
		The expression of $T$ above is true since
		\begin{equation*}
			\frac{1+k([f(v)+c]/p)}{2}=\begin{cases}
				0&\text{ if }([f(v)+c]/p)=-k\\
				1&\text{ if }([f(v)+c]/p)=k
			\end{cases}
		\end{equation*}
		and it follows from the definition of $\pi^k(v)$.
		
		Because $f(v)\equiv \varepsilon(v)(\text{mod }p)$, we have 
		\begin{equation*}
			T=\sum_{1\leq \varepsilon(v)\leq n-qp}\sum_{\substack{c\in F\\c\neq \varepsilon(v)\\\varepsilon(v)+c\not\equiv0(\text{mod }p)}}\frac{1+k([\varepsilon(v)+c]/p)}{2}.
		\end{equation*}
		Let $i=\varepsilon(v)$. Then it follows that
		\begin{align*}
			T&=\sum_{i=1}^{n-qp}\sum_{\substack{c=1\\c\neq i\\i+c\not\equiv0(\text{mod }p)}}^{n-qp}\frac{1+k([i+c]/p)}{2}\\
			&=\frac{1}{2}\sum_{i=1}^{n-qp}\sum_{\substack{c=1\\c\neq i\\i+c\not\equiv0(\text{mod }p)}}^{n-qp}[1+k([i+c]/p)]\\
			&=\frac{1}{2}\left[\sum_{i=1}^{n-qp}\sum_{\substack{c=1\\c\neq i\\i+c\not\equiv0(\text{mod }p)}}^{n-qp}1+k\sum_{i=1}^{n-qp}\sum_{\substack{c=1\\c\neq i\\i+c\not\equiv0(\text{\normalfont{mod} }p)}}^{n-qp}([i+c]/p).\right]
		\end{align*}
		\textbf{Subclaim 7.1.} 
		\begin{equation*}
			\sum_{i=1}^{n-qp}\sum_{\substack{c=1\\c\neq i\\i+c\not\equiv0(\text{mod }p)}}^{n-qp}1=(n-qp)(n-qp-1)-\max\{0,2(n-qp)-p+1\}
		\end{equation*}
		
		Observe that
		\begin{align*}
			\sum_{i=1}^{n-qp}\sum_{\substack{c=1\\c\neq i\\i+c\not\equiv0(\text{mod }p)}}^{n-qp}1&=\sum_{i=1}^{n-qp}\sum_{\substack{c=1\\c\neq i}}^{n-qp}1-\sum_{i=1}^{n-qp}\sum_{\substack{c=1\\c\neq i\\i+c\equiv0(\text{mod }p)}}^{n-qp}1\\
			&=(n-qp)(n-qp-1)-\sum_{i=1}^{n-qp}\sum_{\substack{c=1\\c\neq i\\i+c\equiv0(\text{mod }p)}}^{n-qp}1.
		\end{align*}
		
		If $i+c\equiv 0(\text{mod }p)$, then $c\equiv (p-i)(\text{mod }p)$. Assume that $p-i> n-qp$. So, $i<p-(n-qp)$. Since $c\leq n-qp$, we have $i+c<p$. We have a contradiction because  $i+c\equiv 0(\text{mod }p)$. Thus, $p-i\leq n-qp$. Note that $i\leq n-qp$. So, $p-(n-qp)\leq i\leq n-qp$. The condition $i\neq c$ is always satisfied because $2i\not\equiv0(\text{mod }p)$. Hence, there are $n-qp-p+(n-qp)+1=2(n-qp)-p+1$ pairs $(i,c)$ such that $i+c\equiv 0(\text{mod }p)$. Therefore,
		\begin{equation*}
			\sum_{i=1}^{n-qp}\sum_{\substack{c=1\\c\neq i\\i+c\equiv0(\text{mod }p)}}^{n-qp}1=\begin{cases}
				2(n-qp)-p+1&\text{ if }p-(n-qp)\leq n-qp\\
				0&\text{ otherwise}.
			\end{cases}
		\end{equation*}
		Clearly,
		\begin{equation*}
			\max\{0,2(n-qp)-p+1\}=\sum_{i=1}^{n-qp}\sum_{\substack{c=1\\c\neq i\\i+c\equiv0(\text{mod }p)}}^{n-qp}1.
		\end{equation*}
		This proves Subclaim 7.1. 
		
		\textbf{Subclaim 7.2.} 
		\begin{equation*}
			\sum_{i=1}^{n-qp}\sum_{\substack{c=1\\c\neq i\\i+c\not\equiv0(\text{\normalfont{mod} }p)}}^{n-qp}([i+c]/p)=\sum_{\substack{s=2\\s\neq p}}^{n-qp+1}(s-\delta_s)(s/p)+\sum_{\substack{s=n-qp+2\\s\neq p}}^{2(n-qp)}[2(n-qp)-s+1-\delta_s](s/p)
		\end{equation*}
		We use the idea of reindexing by setting $s=i+c$. Since $1\leq i,c\leq n-qp$, we have $2\leq s\leq 2(n-qp)$. The condition $i+c\not\equiv 0(\text{mod }p)$ implies $s\neq p$. Consider the ordered pair $(i,c)$ and let 
		\begin{equation*}
			A_s=\{(i,c):1\leq i,c\leq n-qp, \text{ }s=i+c\}.
		\end{equation*}
		Suppose that $\eta_s=|A_s|$. Observe that $1\leq s-i\leq n-qp$, that is, $s-(n-qp)\leq i\leq s-1$ since $c=s-i$. Thus, $\max\{1,s-(n-qp)\}\leq i\leq \min\{n-qp,s-1\}$ and it follows that
		$$\eta_s=\min\{n-qp,s-1\}-\max\{1,s-(n-qp)\}+1.$$ 
		In addition, if $i\neq c$, then $s\neq 2i$, that is, $s$ is not even. So, we need to remove the ordered pair $(i,c)$ with $c=i$ from $A_s$ which corresponds to $s=2i$ and it will only occur if $s$ is even. With this, we need the expression $\eta_s-\delta_s$ where
			$$\delta_s=\begin{cases}
			1 &\text{ if $s$ is even}\\
			0&\text{ otherwise}.
		\end{cases}$$
		Now, each ordered pair $(i,c)$ contributes to $([i+c]/p)=(s/p)$ except when $i=c$ ($s$ is even) and $i+c\equiv 0(\text{mod }p)$ ($s=p$). Thus, 
		\begin{equation*}
			\sum_{i=1}^{n-qp}\sum_{\substack{c=1\\c\neq i\\i+c\not\equiv0(\text{\normalfont{mod} }p)}}^{n-qp}([i+c]/p)=\sum_{\substack{s=2\\s\neq p}}^{2(n-qp)}(\eta_s-\delta_s)(s/p).
		\end{equation*}
		Furthermore, it is immediate that
		\begin{equation*}
			\eta_s=\begin{cases}
				s-1&\text{ if }2\leq s\leq n-qp+1\\
				2(n-qp)-s+1&\text{ if }n-qp+2\leq s\leq 2(n-qp).
			\end{cases}
		\end{equation*}
		This proves Subclaim 7.2.
		
		Combining Subclaims 7.1 and 7.2, we obtain Claim 7.
		
		This completes the proof of Lemma~\ref{3.2}. 
	\end{proof}
	
	\begin{pro}
		If $q$ is a positive integer, then $\delta(L_{qp}^k(f,p))=q\left(\frac{p-1}{2}\right)-1$ and $\Delta(L_{qp}^k(f,p))=q\left(\frac{p-1}{2}\right)$.
	\end{pro}
	
	\begin{proof}
		Immediately follows from Lemma~\ref{3.2} (\ref{eq1}).
	\end{proof}
	
	\begin{thm}\label{3.3}
		Let $K_n$ be a complete graph of order $n$ and let $q=\left\lfloor\frac{n}{p}\right\rfloor$. Suppose that $\psi=\max\{0,2(n-qp)-p+1\}$ and let $\mathcal{S}=\mathcal{S}_1+\mathcal{S}_2$ where\\
		\begin{equation*}
			\mathcal{S}_1=\sum_{\substack{s=2\\s\neq p}}^{n-qp+1}(s-1-\delta_s)(s/p)\text{ and }\mathcal{S}_2=\sum_{\substack{s=n-qp+2\\s\neq p}}^{2(n-qp)}[2(n-qp)-s+1-\delta_s](s/p)
		\end{equation*}	
		given that
		\begin{equation*}
			\delta_s=\begin{cases}
				1&\text{ if $s$ is even}\\
				0&\text{ otherwise}.
			\end{cases}
		\end{equation*}
		Then $K_n$ is a Legendre cordial graph modulo $p$ if and only if
		\begin{equation}\label{eq3}
			\mathcal{S}=2nq-pq^2-q+\psi
		\end{equation}
		or
		\begin{equation}\label{eq4}
			\mathcal{S}=2nq-pq^2-q+\psi\pm 2.
		\end{equation}
	\end{thm}
	
	\begin{proof}
		Let $f:V(K_n)\to \{1,2,\cd,n\}$ be an arbitrary bijective function and let $L_n^k(f,p)$ be a Legendre graph for $k=-1,1$. Suppose that $V(K_n)=V(L_n^1(f,p))=V(L_n^{-1}(f,p))$. So, $L_n^1(f,p)$ and $L_n^{-1}(f,p)$ are subgraphs of $K_n$. Also, let $G$ be a subgraph of $K_n$ with vertex set $V(G)=V(K_n)$ and edge set $$E(G)=E(K_n)-(E(L_n^1(f,p))\cup E(L_n^{-1}(f,p))).$$
		Now, it is clear that the edges of $L_n^{-1}(f,p)$ and $L_n^1(f,p)$ are labeled by 0 and 1, respectively, under $f_p^*$. Additionally, observe that for any $ab\in E(G)$, $f(a)+f(b)\equiv 0(\text{mod }p)$, which means all edges of $G$ are labeled by 0 under $f_p^*$. This implies that
		\begin{align*}
			e_{f_p^*}(0)&=|E(L_n^{-1}(f,p))|+|E(G)|\\
			&=|E(L_n^{-1}(f,p))|+|E(K_n)|-(|E(L_n^{-1}(f,p))|+|E(L_n^{1}(f,p))|)\\
			&=|E(K_n)|-|E(L_n^{1}(f,p))|
		\end{align*}
		and $e_{f_p^*}(1)=|E(L_n^{1}(f,p))|$. Note that the size of $K_n$ is $\frac{n(n-1)}{2}$. By Lemma~\ref{3.2},
		\begin{align*}
			e_{f_p^*}(0)-e_{f_p^*}(1)&=|E(K_n)|-2|E(L_n^{1}(f,p))|\\
			&=\frac{n(n-1)}{2}-2\left[\frac{1}{4}(n^2-n-2nq+pq^2+q-\psi+\mathcal{S})\right]\\
			&=\frac{1}{2}[2nq-pq^2-q+\psi]-\frac{1}{2}\mathcal{S}
		\end{align*}
		for all bijective function $f$.
		
		If $K_n$ is a Legendre cordial graph, then  $|e_{f_p^*}(0)-e_{f_p^*}(1)|\leq 1$, and so (\ref{eq3}) and (\ref{eq4}) follow. 
		
		The converse is clear.
	\end{proof}
	The following corollary is a direct consequence of Theorem~\ref{3.3}.
	\begin{cor}\label{3.4}
		The complete graph $K_{qp}$, where $q$ is a positive integer, is a Legendre cordial graph modulo $p$ if and only if $q=1$ and $p=3$.
	\end{cor}
	
	\newpage
	
	\tab The following algorithm is created to determine whether a complete graph admits a Legendre cordial labeling modulo $p$ based on Theorem~\ref{3.3}.

		\begin{algorithmic}[1]
			\STATE \textbf{Input:} $n$ (order of $K_n$), $p$ (odd prime)
			
			\STATE \textbf{function} $\texttt{LegendreSymbol}(a,p)$ (\textit{Euler's Criterion}\cite{Burton})
			\STATE \quad $r \gets a^{\frac{p-1}{2}} \bmod p$ 
			\STATE \quad \textbf{if} $r=1$
			\STATE \quad \quad return $1$
			\STATE \quad \textbf{else}
			\STATE \quad \quad return $-1$
			\STATE \quad \textbf{end if}
			\STATE \textbf{end function}
			
			\STATE \textbf{function} $\delta(s)$
			\STATE \quad $t\gets s \bmod 2$
			\STATE \quad \textbf{if} $t=0$
			\STATE \quad \quad return $1$
			\STATE \quad \textbf{else}
			\STATE \quad \quad return $0$
			\STATE \quad \textbf{end if}
			\STATE \textbf{end function}
			
			\STATE $q \gets \lfloor n / p \rfloor$
			\STATE $\psi \gets \max(0, 2 \cdot (n - qp) - p + 1)$
			
			\STATE $S_1 \gets 0$
			\FOR{$s = 2$ to $n - qp + 1$}
			\IF{$s \neq p$}
			\STATE $L \gets \texttt{LegendreSymbol}(s, p)$
			\STATE $S_1 \gets S_1 + (s - 1 - \delta(s)) \cdot L$
			\ENDIF
			\ENDFOR
			
			\STATE $S_2 \gets 0$
			\FOR{$s = n - qp + 2$ to $2 \cdot (n - qp)$}
			\IF{$s \neq p$}
			\STATE $L \gets \texttt{LegendreSymbol}(s, p)$
			\STATE $S_2 \gets S_2 + (2 \cdot (n - qp) - s + 1 - \delta(s)) \cdot L$
			\ENDIF
			\ENDFOR
			
			\STATE $S \gets S_1 + S_2$
			\STATE \textbf{Check for Legendre cordiality:}
			\STATE $T\gets 2nq-pq^2-q+\psi$
			
			\IF{$S = T$ \OR $S = T \pm 2$}
			\STATE \texttt{LegendreCordial} $\gets$ true
			\ELSE
			\STATE \texttt{LegendreCordial} $\gets$ false
			\ENDIF
			
			\STATE \textbf{Output:} \texttt{LegendreCordial}
		\end{algorithmic}
		
	\tab Suppose that $\mathbb{J}(n,m) = \{p:\text{$K_n$ is a Legendre cordial graph modulo $p$, }p\leq m\}$ and let
	$J(n,m)=|\mathbb{J}(n,m)|$. It is clear that $J(n,m_1)\leq J(n,m_2)$ whenever $m_1\leq m_2$. As a consequence, if $\lim_{n\to \infty} J(n,m_1)=0$, then for any integer $m_2\leq m_1$, we have $\lim_{n\to \infty} J(n,m_2)=0$. Furthermore, it is obvious that $J(2,m)=\pi(m)-1$ where $\pi(m)$ denotes the number of primes not exceeding $m$.
	
	\begin{figure}[hbt!]
		\centering
		\includegraphics[height=3.3in,width=6.15in]{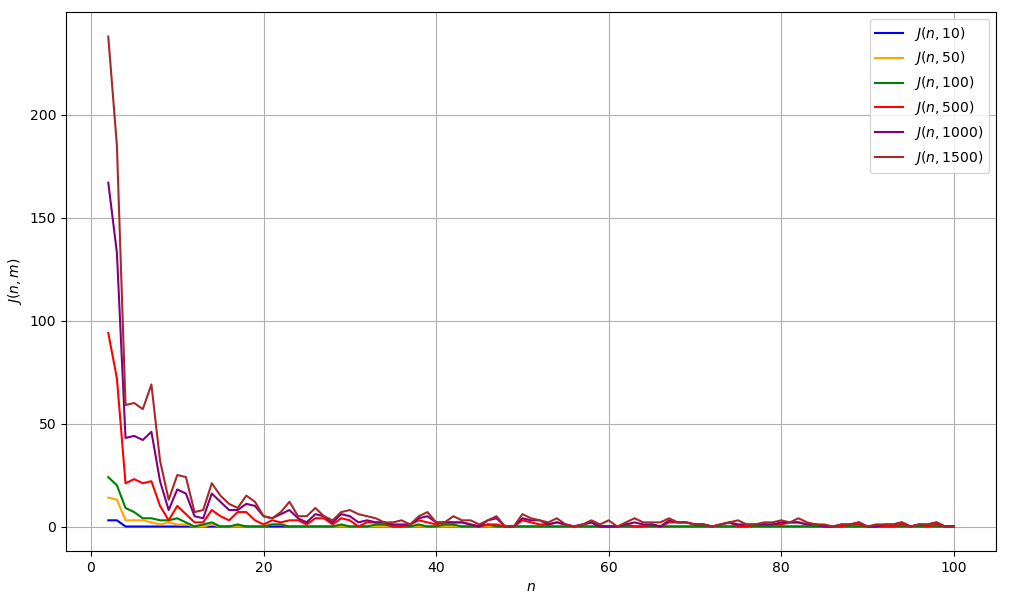}
		\caption{Line plot of $J(n,m)$ for different $m$, $2\leq n\leq 100$}\label{fig1}
	\end{figure}

	\begin{figure}[hbt!]\centering
		\includegraphics[height=3.3in,width=6in]{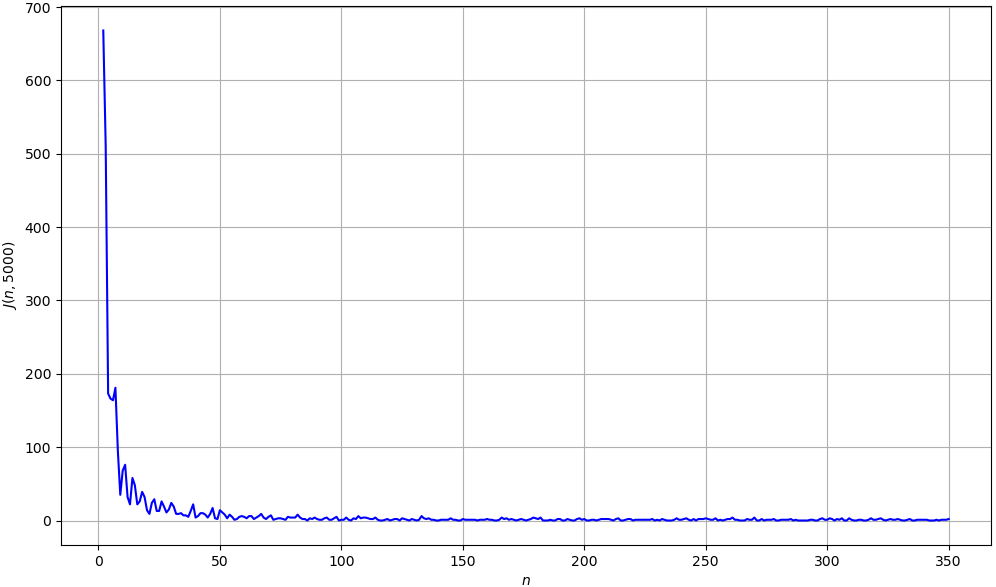}
		\caption{Line plot of $J(n,5000)$ for $2\leq n\leq 350$}\label{fig3}
	\end{figure}
	
	\tab Figure~\ref{fig1} illustrates the behavior of $J(n,m)$ for $m = 10, 50, 100, 500, 1000, 1500$ where $2\leq n\leq 100$. This plot was obtained using the algorithm provided above. The values of $J(n,m)$ fluctuate but show an overall decreasing tendency as $n$ increases. The computational evidence therefore indicates that the number of primes for which Legendre cordial labeling exists becomes very small as the order of complete graph grows.
	
	\tab In addition, it is also worth noting that the same overall trend persists when the computations are extended beyond $n = 100$ and $m=1500$. In particular, as shown in Figure~\ref{fig3}, for $2 \leq n \leq 350$ and $m=5000$, the values of $J(n,m)$ continue to show a downward tendency, confirming the consistency of the observed pattern across larger graph orders, further supporting the observed trend.
	
	\tab Although these figures do not establish a proof, they provide strong empirical support for a broader trend: as the order of the complete graph increases, the likelihood of admitting Legendre cordial labeling diminishes. This observed behavior, together with the theoretical restrictions established in Corollary~\ref{3.4}, motivates the formulation of a general conjecture, stated below.
	
	\begin{con}\label{3.6}
		For any integer $m$, $\lim_{n\to \infty} J(n,m)=0$.
	\end{con}
	
	\section{Conclusion}
	\tab In this paper, we investigated the characterization of Legendre cordial labeling of complete graphs by employing the concept of the Legendre graph. We established necessary and sufficient conditions under which a complete graph admits a Legendre cordial labeling modulo $p$. To complement the theoretical results, a computational algorithm was developed to determine whether a complete graph of order $n$ admits a Legendre cordial labeling modulo $p$. Using this algorithm, values of $J(n,m)$ were generated and analyzed for several choices of $m$. The line plots in Figures~\ref{fig1} and \ref{fig3} consistently show an overall decreasing tendency, providing empirical support for the theoretical restrictions. The combination of analytical results and computational evidence naturally leads to a general conjecture: the number of primes for which a complete graph admits Legendre cordial labeling tends to zero as the order of the graph grows. Future work may consider extending these results to other graph classes, such as bipartite or multipartite graphs, as well as exploring potential applications in areas where number theory and graph labeling intersect, including cryptography and coding theory.

	\end{document}